\documentclass[12pt,a4paper,leqno]{article}
\usepackage[margin=2cm]{geometry}

\usepackage[normalem]{ulem}

\usepackage{amsmath,amsfonts,amssymb,amscd,amsthm}
\usepackage{hyperref}
\usepackage{cleveref}
\usepackage[shortlabels]{enumitem}
\usepackage{url}
\usepackage{tikz-cd}
\usepackage{graphicx}
\usepackage{comment}
\usepackage{mathtools} 
\usepackage{listings} 
\usepackage[ruled,vlined]{algorithm2e}

\usepackage[style=alphabetic,bibstyle=alphabetic,citestyle=alphabetic,backend=biber,doi=false, isbn=false, maxnames=99]{biblatex} 

\DeclareLanguageMapping{british}{british-apa} 

\numberwithin{equation}{section}
\newtheorem{Theorem}[equation]{Theorem} 
\newtheorem{Proposition}[equation]{Proposition} 
\newtheorem{Question}[equation]{Question} 
\newtheorem{Lemma}[equation]{Lemma} 
\newtheorem{Corollary}[equation]{Corollary}

\newtheorem{Prob}[equation]{Problem} 
\newtheorem{Example}[equation]{Example}

\newtheorem{Remark}[equation]{Remark} 
\theoremstyle{definition}
\newtheorem{Definition}[equation]{Definition}

\newtheoremstyle{named}{}{}{\itshape}{}{\bfseries}{.}{.5em}{#3}
\theoremstyle{named} 

\numberwithin{table}{section}

\newcommand{\ol}{\overline}

\newcommand{\R}{\mathbb R}
\newcommand{\Q}{\mathbb Q}
\newcommand{\Z}{\mathbb Z}
\newcommand{\F}{\mathbb F}

\newcommand{\N}{\mathbb N}

\newcommand{\A}{\mathbb A}

\newcommand{\calO}{\mathcal{O}}

\newcommand{\calR}{\mathcal{R}}

\newcommand{\Rcal}{\mathcal{R}}
\newcommand{\calT}{\mathcal{T}}
\newcommand{\calG}{\mathcal{G}}

\newcommand{\Syg}{\mathcal{S}}

\newcommand{\lto}{\longrightarrow}
\newcommand{\into}{\hookrightarrow}

\DeclareMathOperator{\Spec}{Spec}

\newcommand{\Hom}{\mathrm{Hom}}

\newcommand{\rank}{\mathrm{rank}}
\renewcommand{\phi}{\varphi}

\newcommand{\LT}{\mathrm{LT}}
\newcommand{\ec}{\mathrm{ec}}

\DeclareMathOperator{\pfrak}{\mathfrak{p}}
\DeclareMathOperator{\mfrak}{\mathfrak{m}}

\DeclareMathOperator{\Tcal}{\mathcal{T}}

\DeclareMathOperator{\Fbb}{\mathbb{F}}

\title{A multivariate  Strassmann theorem }

\author{Guido Maria Lido \& Luca Mauri \\ 
	\small{
		\qquad  \href{mailto:guidomaria.lido@gmail.com}{guidomaria.lido@gmail.com}} 	\quad \href{mailto:luca.mauri@phd.unipi.it}{luca.mauri@phd.unipi.it} \\
		\small{Universit\`{a} di Roma Tor Vergata} \quad \small{Universit\`{a} di Pisa} \qquad\qquad  
        }
        
\date{}

\begin{document}

\maketitle

\begin{abstract}
By a theorem of Strassmann, a non-zero convergent power series in one variable over a complete non-Archimedean field has finitely many zeros, with an explicit bound on their number. 
We generalize this result to convergent power series in several variables, characterizing finiteness of the zero set and bounding its cardinality in terms of the reduction of the saturated ideal defined by the power series.

We discuss how to make our result effective, under suitable assumptions, when working with approximate power series.
\end{abstract}


\section{Introduction}
\label{section introduction}

Let $K$ be a field complete with respect to a non-trivial, non-Archimedean (possibly not discrete) valuation $v \colon K \to \R \cup \{+\infty\}$. Its ring of integers $R\coloneqq\{a\in K \,\colon\, v(a)\ge0\}$ has Krull dimension $1$, since $v$ is non-trivial and takes values in $\R \cup \{+\infty\}$.  We denote  by $\mfrak\coloneqq\{a\in K \,\colon\, v(a)>0\}$ the unique maximal ideal of $R$ and by $k\coloneqq R/\mfrak$  the residue field of $\mfrak$. 

A power series in $K[[t]]$ converges on $R$ if and only if the valuation of its coefficients tends to $+\infty$. In \cite{MR1581148}, Strassmann proved that a non-zero convergent power series has only finitely many zeros in $R$, and gave an explicit upper bound.\footnote{Strassmann focused on the case where $K$ is a $p$-adic field, but the result holds in greater generality}

\begin{Theorem}[Strassmann]
\label{Strassmann's Theorem section introduction}
    Let 
    \begin{equation*}
        f(x)=\sum_{n=0}^{\infty}a_{n}x^{n} 
    \end{equation*}
    be a non-zero power series with coefficients in $K$ that converges on $R$.
    Then $f$ has only finitely many zeros in $R$ and, defining $M=\min_{n}v(a_n)$, the number of zeros of $f$ in $R$ is at most the natural number
    \begin{equation} \label{eq:strassman bound}
        \max\{n \,\colon\, v(a_{n})=M\} .
    \end{equation}
\end{Theorem}

The above finiteness result can be applied to prove basic properties of $p$-adic analytic functions, e.g., that a convergent power series $f$ on $R$ is determined by the values $f(a_i)$ at infinitely many points $a_i \in R$, and that it is not periodic unless it is constant.
Moreover, Strassmann's bound can sometimes be used to bound the number of solutions of Diophantine equations. In particular, it comes in handy when $p$-adic methods can be applied, such as Skolem's method for Thue equations (\cite{skolem1934verfahren}) or Chabauty-Coleman's method for rational points on curves of genus at least~$2$ (\cite{MR4484,Coleman}). Another noteworthy consequence is the Skolem--Mahler--Lech Theorem, which concerns the zero set of a linear recurrence with constant coefficients in a field of characteristic zero (\cite{skolem1934verfahren, zbMATH03015968, lech1953note}).

We present a generalization of Strassmann's theorem 
to the setting of convergent power series in several variables. The set of formal power series in $n$ variables that converge in the unit ball $R^{n}$ is the Tate algebra
\begin{equation}\label{eq:T_n}
    \mathcal{T}_{n}=K\langle x_{1},\,\dots,\,x_{n}\rangle\coloneqq\biggl\{\sum_{I\in\N^{n}} a_{I}x^{I}\,\colon\,a_{I}\in K \text{ and } v(a_{I})\to +\infty \text{ as } |I|\to +\infty\biggr\} .
\end{equation}
We also consider the algebra of its ``integral" elements
\begin{equation}\label{eq:R_n}
    \mathcal{R}_{n}=R\langle x_{1},\,\dots,\,x_{n}\rangle\coloneqq\biggl\{\sum_{I\in\N^{n}} a_{I}x^{I}\,\colon\,a_{I}\in R \text{ and } v(a_{I})\to +\infty \text{ as } |I|\to +\infty\biggr\} .
\end{equation}
A power series  $f$ in $\calR_n$ can be reduced modulo $\mfrak$ by reducing its coefficients, and the convergence condition ``$v(a_{I})\to +\infty$'' implies that the reduction $\ol f$ is in fact a polynomial in $k[x_1,\ldots, x_n]$. 
The same convergence condition also implies that for any non-zero power series $f \in \calT_n$ there exists a constant $c \in 
K^\times$ such that $c\cdot f$ lies in $\calR_n$. Moreover, choosing carefully the valuation of $c$, we can suppose that the reduction $\ol{c\cdot f}$ is non-zero. For such a constant $c$, Strassmann's bound in Equation \eqref{eq:strassman bound} equals the degree of $\ol{c f}$ or, equivalently, the dimension of the $k$-vector space $k[x]\bigl/ (\ol{cf})$. Our main result is a direct generalization of this observation. 
Before stating the theorem, for any subset $I \subset \calT_n$, let us introduce the notation 
\begin{equation}
    V_R(I) :=  \big\{ (a_1, \ldots, a_n) \in R^n : f(a_1, \ldots, a_n) = 0 \quad \forall f\in I \big\},
\end{equation}
to denote the set of zeros of $I$ in $R^n$.

\begin{Theorem}
\label{main theorem section introduction}   
    Let $f_{1}, \dots, f_{r}$ be power series in $\Rcal_{n}$ and let $\ol{f_{1}}, \dots,\ol{f_{r}}$ be their reductions modulo~$\mfrak$. If
    \begin{equation*} \label{eq:defAbar}
       \ol{A}\coloneqq k [ x_1,\ldots, x_n ]/ (\ol{f_1}, \ldots, \ol{f_r})
    \end{equation*}
    is a finite-dimensional $k$-vector space, then the set 
    of common zeros of $f_{1},\dots,f_{r}$
    in $R^{n}$ is finite and satisfies the inequalities
    \begin{equation}\label{eq:bound_zeroes}
       \# V_{R}\bigl(\{ f_{1}, \dots, f_{r}\}\bigr) \le \sum_m \dim_k(\ol{A}_m) \le \dim_k(\ol{A}) ,
    \end{equation}
    where the sum ranges over the finitely many maximal ideals~$m \subset \ol A$ such that $\ol{A}/m=k$, and $\ol{A}_{m}$ denotes the localization of $\ol{A}$ at $m$.
\end{Theorem}

The above theorem is a more general version of \cite[Theorem~4.12]{MR4556934}. 
We provide a different proof, which also applies when the valuation is not discrete. 

In the context of Theorem \ref{main theorem section introduction}, the set of zeros $V_R(\{f_1,\ldots, f_r\})$ is equal to $V_R(I)$, where $I \subset \calR_n$ is the ideal generated by the $f_i$. Then $\ol A$ is the reduction of $\calR_n/ I$.   
In Section~\ref{section main results}, we show that the finiteness of $\ol{A}$ is  necessary when $K$ is algebraically closed and $I$ is replaced by its saturation
\begin{equation}
\label{equation saturation of an ideal section introduction}
    I^{\ec} := \bigl\{ f\in \calR_n : \exists r\in R \setminus \{0\} \, \text{ such that }\, rf \in I \bigr\} .
\end{equation}

We then study the computation of such a saturation when the generators $f_i$ of $I$ are known with finite precision, as in arithmetic applications.

In this context, we use the work of Caruso, Vaccon and Verron in \cite{MR4007445, MR4398768, MR4488860}, who define Gr\"obner bases for ideals of convergent power series, studying their properties and proposing  algorithms. 
Gröbner bases allow us to compute the saturation \eqref{equation saturation of an ideal section introduction} in certain cases; moreover, the two notions are closely related (see Remark~\ref{rmk:saturation_vs_Grobner}).

We also highlight the work of Rabinoff \cite{MR2900439}, which uses the theory of toric varieties and tropical geometry to study the possible valuations of the zeros of a system of power series. Moreover, Theorem~$11.7$ in loc.\ cit. gives, in some cases, a formula for the multiplicities of isolated zeros with given valuation in terms of mixed volumes, generalizing the Bernstein–Khovanskii–Kushnirenko theorem (cf. \cite{MR435072, MR419433}). 
See also \cite{MR3279033} and \cite{KRZB} for arithmetic applications.

We conclude this introduction by describing the structure of the paper. In Section~\ref{section main results} we prove Theorem~\ref{main theorem section introduction} using Theorem~\ref{Theorem main}, a dimensionality result holding for all ideals. In Section~\ref{section saturation of an ideal} we describe a procedure to compute the saturation \eqref{equation saturation of an ideal section introduction} of an ideal. Finally, in Section~\ref{section example} we illustrate an application to a Thue equation, by first applying  Skolem's method to bound its solutions with the zeros of a system of power series in two variables, and then bounding such zeros with our results.

\subsubsection*{Acknowledgments} 

We have the pleasure to thank Umberto Zannier and Davide Lombardo for posing the question that started this work and Giovanni Marzenta for the useful discussions.

The first author is supported by the MIUR Excellence Department Project MatMod@TOV awarded to the Department of Mathematics, University of Rome Tor Vergata, by the ``Programma Operativo Nazionale (PON) ``Ricerca e Innovazione''  2014-2020, 
and the PRIN PNRR 2022 ``Mathematical Primitives for Post Quantum Digital Signatures''.

Both authors are supported by the ``National Group for Algebraic and Geometric Structures, and their Applications" (GNSAGA - INdAM).

\section{Main results}
\label{section main results}

We keep the notation as in Section~\ref{section introduction}.
Before stating the main results, we recall the Weierstrass division and preparation theorems in the context of $\calR_n$. These results are more often stated for power series in $\Tcal_{n}$, see \cite[Section~$2.2$, Lemma~$7$, Theorem~$8$, and Corollary~$9$]{MR3309387}.

\begin{Theorem}
\label{Weierstrass division and preparation theorem}
    Let $f$ be a non-constant power series in $\Rcal_{n}$. Then there exist an automorphism $\sigma$ of $\Rcal_{n}$, a constant $a\in R$, a unit $u\in\Rcal_{n}$, and a monic polynomial $\tilde{f}\in\Rcal_{n-1}[x_{n}]$ such that $\sigma(f)=a\cdot u\cdot\tilde{f}$.

    Suppose $f$ is a monic polynomial in $\Rcal_{n-1}[x_{n}]$ of degree $d$, and let $g$ be a power series in $\Rcal_{n}$. Then there exist a unique power series $q\in\Rcal_{n}$ and a polynomial $r\in\Rcal_{n-1}[x_{n}]$ of degree $<d$ such that $g=qf+r$.
\end{Theorem}
\begin{proof}
    See \cite[Theorems 2.1, 3.1 and 4.1]{MR607075}.  
\end{proof}

A consequence of this theorem is the computation of the Krull dimension of $\calR_n$.

\begin{Proposition}
\label{proposition Krull dimension of Kn}
    The ring $\Rcal_{n}$ has Krull dimension $n+1$.
\end{Proposition}
\begin{proof}
    The chain of prime ideals
    \begin{equation*}
        (0) \subsetneq (x_{1}) \subsetneq (x_{1},\,x_{2}) \subsetneq \cdots \subsetneq (x_{1},\,x_{2},\,\dots,\,x_{n}) \subsetneq (\mfrak,\,x_{1},\,x_{2},\,\dots,\,x_{n}) ,
    \end{equation*}
    implies that $\dim \Rcal_n \ge n+1$. For the other inequality we proceed by induction on $n$. In the case $n=0$, we have $\Rcal_0 = R$ which is local of Krull dimension $1$. For the inductive step we assume the statement is true for $\Rcal_{n}$ and we prove it for $\Rcal_{n+1}$. It is enough to prove that, for each non-zero prime ideal $\pfrak$ of $\Rcal_{n+1}$, the quotient $\Rcal_{n+1}/\pfrak$ has dimension at most $n+1$.
    
    If $\pfrak\,\cap\, R \neq (0)$ we necessarily have $\pfrak\,\cap\, R = \mfrak$, implying
    \[
    \dim(\Rcal_{n+1}/\pfrak) \le \dim (\Rcal_{n+1}/\mfrak\cdot \Rcal_{n+1}) = \dim \bigl(k[x_1, \ldots, x_{n+1}]\bigr) = n+1 .
    \]
    Now suppose $\pfrak\,\cap\, R = (0) $. Since $\pfrak$ is non-zero, it contains a non-constant power series $f$. By the first part of Theorem \ref{Weierstrass division and preparation theorem}, up to automorphisms of $\calR_{n+1}$ and up to multiplying $f$ by a unit, we may assume that $f = a \tilde f$, where $a$ is a non-zero constant in $R$ and $\tilde f$ is a monic polynomial in $\Rcal_{n}[x_{n+1}]$. Since $\pfrak$ is a prime ideal and $\pfrak\,\cap\, R = (0) $, we deduce $\tilde{f}\in\pfrak$. By the second part of Theorem \ref{Weierstrass division and preparation theorem}, the ring $\Rcal_{n+1}/(\tilde f)$ is a finite extension of $\Rcal_{n}$, hence it has the same dimension. This, together with the inductive hypothesis, implies
    \[
    \dim(\Rcal_{n+1}/\pfrak) \le \dim \bigl(\Rcal_{n+1}/(\tilde f)\bigr) = \dim (\Rcal_{n}) = n+1 .
    \]
\end{proof}

Given an ideal $I$ of $\Rcal_{n}$, define its extended ideal as
\begin{equation}
\label{equation definition extended ideal section main results}
    I^{\mathrm{e}}\coloneqq I\cdot\Tcal_{n},
\end{equation}
so that the saturation of $I$ defined in \eqref{equation saturation of an ideal section introduction} can also be described as
\begin{equation*}
    I^{\ec} = I^{\mathrm{e}} \cap \Rcal_{n} .
\end{equation*}
We say that $I$ is saturated if $I=I^{\ec}$. Notice that
\[
V_R(I) = V_R(I^{\mathrm e}) = V_R (I^\ec).
\]

\begin{Remark}\label{rem:pointsandHom}
    The zeros $a$ in $R^{n}$ of an ideal $I \subset \calR_n$ are in bijection with $K$-algebra  morphisms $\Tcal_n/I^{\mathrm{e}} \to K$, with $R$-algebra morphisms $\Rcal_n/I \to R$ and with $R$-algebra morphisms $\Rcal_n/I^{\ec} \to R$. 
    Indeed a zero $a\in V_R(I)$ induces maps
    \begin{align*}
        \Rcal_{n} &\lto R , \quad f \longmapsto f(a_{1}, \, \dots, \, a_{n}) ,
        \qquad  \qquad
        \Tcal_{n} 
        \lto K , \quad  f \longmapsto f(a_{1}, \, \dots, \, a_{n}) ,
    \end{align*}
    with kernel containing $I^{\ec} \supset I$ in the first case, $I^{\mathrm{e}}$ in the second. All morphisms $\calT_n \to K$ are of this form by \cite[Section~$2.2$, Corollary~$13$]{MR3309387} (in particular all such morphisms are continuous), and the same is true for morphisms $\calR_n \to R$, since they extend to morphisms $\calR_n\otimes_R K = \calT_n \to K$.
    
    With the same motivation, if $L$ is any valued field extending $K$ and $\calO_{L}$ is its ring of integers, 
    then continuous $R$-algebra  morphisms $\Rcal_n/I \to \calO_{L}$ and continuous $K$-algebra  morphisms $\Tcal_n/I^{\mathrm{e}} \to L$  are in bijection with $V_{\calO_L}(I)$, that is, common zeros of $I$ in $\calO_L^n$.  
\end{Remark}

The following theorem provides a dimensional criterion for an ideal of $\Rcal_{n}$ to have finitely many common zeros in the valuation ring of a completion of an algebraic closure of $K$. We take the completion to ensure that automorphisms of $\calT_{n}$ induce bijections between zeros of ideals. For any saturated ideal $I \subset \calR_n$, the theorem also relates the dimension of $\calR_n/I$ with the dimension of the algebraic variety defined by its reduction modulo $\mfrak$.

\begin{Theorem}
\label{Theorem main}

Let $I$ be an ideal of $\Rcal_{n}$, let $I^{\ec}$ be its saturation, let $I^{\mathrm{e}}$ be its extended ideal in $\Tcal_{n}$, and let $\ol{I^{\ec}}$ be the reduction modulo $\mfrak$ of the saturation. If $I\cap R=(0)$, then
\begin{equation}
\label{equation dimension of quotient of Tn, Kn and Ln of main Theorem}
\dim\left(\Tcal_{n}/I^{\mathrm{e}}\right)
=\dim\bigl(k[x_1,\, \dots,\,x_{n}]/\ol{I^{\ec}}\bigr) = \dim\bigl(\Rcal_{n}/I^{\ec}\bigr)-1 .
\end{equation}
Moreover, for all ideals $I$ of $\Rcal_{n}$, the following conditions are equivalent:
\begin{enumerate}[label=\arabic*)]
\label{enumerate theorem main general finiteness}
\item \label{item1} the power series in $I$ have finitely many common zeros in $(\calO_{K^c})^n$, where $K^{c}$ is a completion of an algebraic closure of $K$ and $\calO_{K^c}$ is its ring of integers; 
\item \label{item2} the subvariety of $\A^n_k$ defined by $\overline{I^{\ec}}$ has dimension zero or is empty;
\item \label{item3} the $k$-vector space $k[x_1, \ldots, x_n]/\overline{I^{\ec}}$ is finite-dimensional;
\item \label{item4} the $R$-module $\Rcal_{n}/I^{\ec}$ is finitely generated;
\item \label{item5} the $K$-vector space $\Tcal_{n}/I^{\mathrm{e}}$ is finite-dimensional.
\end{enumerate}
\end{Theorem}
\begin{proof}
We prove \eqref{equation dimension of quotient of Tn, Kn and Ln of main Theorem} by induction on $n$. In the case $n=0$, we have $\Rcal_n = R$, $\Tcal_n = K$ and $I = (0) = I^{\ec}$, hence \eqref{equation dimension of quotient of Tn, Kn and Ln of main Theorem} is true. We now assume that Equation~\eqref{equation dimension of quotient of Tn, Kn and Ln of main Theorem} holds for ideals of $\Rcal_{n-1}$ and we prove that it holds for $I \subset \Rcal_n$. 

If $I=(0)$, then the statement coincides with Proposition~\ref{proposition Krull dimension of Kn}, together with \cite[Chapter~2, Proposition~17]{MR3309387}. If $I \neq (0)$, then we pick a non-zero element $f \in I$, which is non-constant by hypothesis. Since an automorphism of $\Rcal_{n}$ induces an automorphism of $\Tcal_{n}$ and of $k[x_1, \, x_{2}, \, \dots, \, x_n]$, and since a ring automorphism preserves the Krull dimension, by the first part of Theorem~\ref{Weierstrass division and preparation theorem} we can suppose that $f = a \cdot \tilde f$, where $a$ is a non-zero element of $R$ and $\tilde f$ is a monic polynomial in $\Rcal_{n-1}[x_n]$. By the  definition in~\eqref{equation saturation of an ideal section introduction}, we deduce that $\tilde{f}$ belongs to $I^{\ec}$.
The second part of Theorem~\ref{Weierstrass division and preparation theorem} tells us that the ring extension 
$\Rcal_{n-1} \subset \Rcal_n/(\tilde f) $  is finite.  
In particular, the following ring extensions are also finite
\begin{equation} 
\label{eq:finite_ext_0}
\begin{aligned}
\Rcal_{n-1}/I^{\ec}_{n-1} &\subseteq \Rcal_{n}/I^{\ec} , & &\text{where } I^{\ec}_{n-1} \coloneqq I^{\ec}\cap\Rcal_{n-1} ,  \\
\Tcal_{n-1}/I^{\mathrm{e}}_{n-1} &\subseteq \Tcal_{n}/I^{\mathrm{e}} 
, & &\text{where } I^{\mathrm{e}}_{n-1} \coloneqq I^{\mathrm{e}} \cap \Tcal_{n-1} ,  \\
k[x_1, \ldots, x_{n-1}]/ J &\subseteq k[x_1, \ldots, x_{n}] /\ol{I^{\ec}} , & &\text{where } J \coloneqq \ol{I^{\ec}}\cap k[x_1, \ldots, x_{n-1}] .
\end{aligned}
\end{equation}

Since finite extensions preserve Krull dimension it is enough to prove that 
\begin{equation}  \label{eq:inductive_dims}
\dim\left(\Tcal_{n-1}/I^{\mathrm{e}}_{n-1}\right)
= \dim\bigl(\Rcal_{n-1}/I^{\ec}_{n-1}\bigr)-1 =\dim\bigl(k[x_1,\,\dots,\,x_{n-1}]/J \bigr) 
.
\end{equation}
Since  $I^{\ec}$ is saturated, also $I^{\ec}_{n-1}$ is saturated, and its extension in $\calT_{n-1}$ is equal to $I^{\mathrm{e}}_{n-1}$. In particular, the first equality of \eqref{eq:inductive_dims} is true by inductive hypothesis. We are left to prove the second equality which, using the inductive hypothesis, is equivalent to 
\begin{equation}  \label{eq:J_inductive}
\dim\bigl(k[x_1,\,\dots,\,x_{n-1}]/J \bigr) 
= \dim\left(k[x_1,\,\dots,\,x_{n-1}]/\ol{I^{\ec}_{n-1} }  \right) .
\end{equation}
By definition we have $\ol{I^{\ec}_{n-1} } \subset J$, which gives one inequality. For the other inequality, we prove that the radical of $J$ is contained in the radical of $\ol{I^{\ec}_{n-1}}$. We prove it by showing that $\ol{I^{\ec}_{n-1}}$ contains all the polynomials $g^d$, with $g$ in $J$ and $d$ the degree of $\tilde f$.  By definition, if $g \in k[x_1, \ldots, x_{n-1}]$  lies in $J$, then it is the reduction of some $\tilde g \in  I^{\ec}$ and, up to applying Theorem \ref{Weierstrass division and preparation theorem} and taking the remainder modulo $\tilde f$, we can suppose that $\tilde g$ is a polynomial in $x_n$ of degree $e<d$ (indeed for any power series $g' \in \calR_n$ reducing to $g$, if we apply Theorem \ref{Weierstrass division and preparation theorem} and we write $g' =  q\tilde f + r$, taking the reduction and comparing the degree in $x_n$ we deduce that $q$ reduces to $0$ modulo~$\mfrak$, hence $r$ reduces to $g$). 
Then, the resultant $h$ of $\tilde g$ and $\tilde f$, considered as polynomials in $\calR_{n-1}[x_n]$, is an element of $I^{\ec}_{n-1}$: indeed $h$ lies $\calR_{n-1}$ and, by \cite[Chapter IV, $\S8$, start of page 202]{LangAlgebra}, it is a $\calR_n$-linear combination of $\tilde f$ and $\tilde g$. 
Moreover, using the definition of the resultant, we can write  $h$ as the determinant of a matrix with entries depending on the coefficients of $\tilde f$ and $\tilde g$: since $\tilde f$ is monic of degree~$d$ and $\tilde g = a_{e}x_n^{e} + \ldots + a_0 $ with $a_i\in \calR_{n-1}$ such that $a_0 \equiv g \pmod \mfrak$ and $a_i\equiv 0 \pmod \mfrak$ for $i\ge 1$, we see that $h \equiv a_0^d \equiv g^d \pmod \mfrak$. In particular  $\ol{I^{\ec}_{n-1}}$ contains $\ol h =  g^d$. 

We now prove the equivalence of conditions \ref{item1} ... \ref{item5}. 
If $I$ contains a constant $r\in R\setminus\{0\}$, then all conditions are satisfied, hence we can suppose that $I \cap R = (0)$.

We start from the implication $1)\implies  2)$. 
Again we proceed by induction on $n$. 
When $n=1$, we have two cases. If $I=(0)$ both conditions are false. Otherwise, by Theorem \ref{Weierstrass division and preparation theorem}, we can suppose that $I$ contains an element $a \cdot \tilde f$ with $a\in R \setminus \{0\}$ and $\tilde f$ a monic polynomial in $R[x_1]$. Then, $\tilde f$ belongs to $I^{\ec}$ and both conditions are true since polynomials have finitely many roots.

For the inductive step, we can assume that the ring extensions in \eqref{eq:finite_ext_0} are finite.
In particular, using Equation \eqref{eq:J_inductive}, the subvariety of $\A^n_k$ defined by $\ol{I^{\ec}}$ has the same dimension as the subvariety of $\A^{n-1}_k$ defined by $\ol{I^{\ec}_{n-1}}$.  
Assume by contradiction that the former has dimension at least $1$. Then so does the latter, so by the inductive hypothesis $I^{\ec}_{n-1} \subset \Rcal_{n-1}$ has infinitely many common zeros $z = (z_1,\ldots, z_{n-1}) \in (\calO_{K^c})^{n-1}$. We show each such zero lifts to a common zero $(z_1,\ldots, z_n)$ of $I^{\ec}$, contradicting the finiteness of $V_{\calO_{K^c}}(I) = V_{\calO_{K^c}}(I^{\ec})$.
For each such $z$, the set $P_z = \{ f \in \Rcal_{n-1}: f(z)=0\}$ induces a prime ideal of $A_{n-1}= \Rcal_{n-1}/I^{\ec}_{n-1}$. By the going-up theorem  \cite[Chapter~2, Theorem~5]{MR1011461}, this extends to a prime ideal $P$ of $A_n =  \Rcal_{n}/I^{\ec}$. By construction, $A_{n-1}/P_z$ embeds into $\calO_{K^c}$ by sending $x_i \mapsto z_i$ for $i = 1, \ldots, n-1$.  Since $A_n /P$ is an integral domain finite over $A_{n-1}/P_z$, this embedding extends to $A_{n}/P  \into \calO_{K^c}$ (indeed we have an embedding of the fraction field of $A_{n-1}/P_z$ in $K^c$ which, since $K^c$ is algebraically closed, can be extended to an embedding of the fraction field of $A_n/P$ in $K^c$ and, by integrality property, the image of $A_n/P$ is contained in $\calO_{K^c}$). Thus, with $z_n$ the image of $x_n$, the point $(z_1, \ldots, z_{n-1}, z_n) \in (\calO_{K^c})^n$ is a zero of $I^{\ec}$.

Now the implication $2)\Longrightarrow 3)$. If the subvariety of $\A^n_k$ defined by $\overline{I^{\ec}}$ has dimension zero, it possesses only a finite number of points, hence the ring $k[x_1, \ldots, x_n]/\overline{I^{\ec}}$ has only a finite number of maximal ideals, which implies that it is Artinian and has finite dimension as a $k$-vector space. 

Now $3)\Longrightarrow 4)$. When the valuation is discrete, the ring $\calR_n$ is Noetherian (see \cite[Chapter~0, Proposition~$7.5.4$]{MR3075000}), and this implication follows from Nakayama's lemma. In the general case, we don't have Noetherianity, but we can proceed by induction on $n$, as for (\ref{equation dimension of quotient of Tn, Kn and Ln of main Theorem}). The case $n=0$ is immediate. In the inductive step, we look at the ring extensions in \eqref{eq:finite_ext_0}, and  we can again suppose that they are finite.
In particular, $k[x_1, \ldots, x_{n-1}] / J$,  being a subspace of $ k[x_1, \ldots, x_{n}]/ \ol{I^{\ec}}$, is finite-dimensional over $k$, hence Artinian, and the same holds for $k[x_1,\,\dots,\,x_{n-1}]/\ol{I^{\ec}_{n-1} }$ by Equation \eqref{eq:J_inductive}. Then, by inductive hypothesis, $\Rcal_{n-1}/I^{\ec}_{n-1}$ is finite over $R$. Since $\Rcal_{n}/I^{\ec}$ is finite over $\Rcal_{n-1}/I^{\ec}_{n-1}$, it is also finite over $R$ as desired. 

The implication $4) \Longrightarrow 5)$ is a base change. 

Finally the implication $5) \Longrightarrow 1)$. 
Since $\calT_n/ I^{\mathrm{e}}$ is finite-dimensional over $K$, it is also Artinian and its quotients by maximal ideals are finite extensions of $K$.
This implies that there are only finitely many ring homomorphisms from $ K \langle x_1,\ldots, x_n \rangle/ I^{\mathrm{e}}$ to $K^c$ and in particular finitely many common zeros in $V_{\calO_{K^c}}(I^{\rm e}) = V_{\calO_{K^c}}(I)$, since by Remark~\ref{rem:pointsandHom} each such zero determines a ring homomorphism.
\end{proof}

In the above theorem,  Item \ref{item3} gives a necessary and sufficient criterion for an ideal of power series to have finitely many zeros, in terms of the reduction of the saturation of the ideal. If we substitute the saturation of the ideal with the ideal itself, the criterion is only sufficient.  
We now prove Theorem \ref{main theorem section introduction}, which provides, under this hypothesis, a bound on the number of zeros, in terms of the reduction of the ideal.

\begin{proof}[Proof of Theorem \ref{main theorem section introduction}]
Since $\ol{A}$ is a finite-dimensional $k$-vector space, it is an Artinian ring and we have 
\begin{equation} \label{eq:decomposiotionAbar}
\ol{A}\cong\prod_{m \in \mathrm{Specm} (\ol A)}\ol{A}_{m} ,
\end{equation}
with the product taken over all the finitely many maximal ideals of $\ol{A}$. This implies the second inequality in \eqref{eq:bound_zeroes}. 

Now define  $A \coloneqq \calR_n/I$, with $I$ the ideal generated by $f_1, \ldots, f_r$, so that $\ol A = A \otimes k$. 
To prove the first inequality in Equation \eqref{eq:bound_zeroes}, we can suppose that $I = I^{\ec}$, since $I$ and $I^{\ec}$ have the same common zeros and if we substitute $\ol A$ with $\ol B = k[x_1, \ldots, x_n]/\ol{I^{\ec}}$ the quantity $ \sum_m \dim_k(\ol{A}_m)$ in Equation \eqref{eq:bound_zeroes} decreases. 
Then, assuming that $I$ is saturated, we prove that the decomposition in \eqref{eq:decomposiotionAbar} lifts to $A$, as in Equation \eqref{eq:A_product_locals}  below.

By the equivalence of conditions \ref{item3} and \ref{item4} in Theorem \ref{Theorem main},  we know that the ring $A$ is finite as an $R$-module. 
Actually $A$ is a free $R$-module since any minimal set of generators is linearly independent:  the hypothesis ``$I$ is saturated'' translates as ``$A$ is a torsion-free $R$-module'', hence if $a_1, \ldots, a_s$ is a minimal set of generators, any non-trivial relation $\sum r_i a_i =0$, with $r_i\in R$ and, say, $v(r_1) \le \ldots \le v(r_s)$, would yield a relation $r_1(a_1 + \sum_{i \ge 2} (r_i/r_1) a_i) =0$, implying that $a_1$ is not necessary as a generator.

Since $A$ is a finite, free $R$-module and $R$ is complete, then, for any element $\pi\in \mfrak$, the ring $A$ is complete with respect to the $\pi$-adic topology, hence it is canonically isomorphic to the projective limit of $A/\pi^n A$. 
We deduce that
\begin{align}\label{eq:A_first_product_locals}
A &= \underset{n}{\varprojlim} \, \,\frac{A}{\pi^n A} = \underset{n}{\varprojlim} \left( \prod_{m \in \mathrm {Specm} (\ol A)} \left(\frac{A}{\pi^n A} \right)_m\right) \\
&= \underset{n}{\varprojlim} \left( \prod_{m \in \mathrm {Specm} (\ol A)} \frac{A_m}{\pi^n A_{m}} \right)  =
\prod_{m \in \mathrm {Specm} (\ol A)} \left( \underset{n}{\varprojlim}  \left(\frac{A_m}{\pi^n A_{m}} \right)\right) ,
\end{align}
where we identify the maximal ideals of $\ol A$ with their inverse images in $A/\pi^n A$ and in $A$. To motivate the second equality, we can argue by Noetherianity of $B = A/\pi^n A$ when the valuation is discrete. When the valuation is not discrete, we notice that the ideal $\mfrak B$ is locally nilpotent, i.e., all its elements are nilpotent, which implies that all the primes of $B$ actually come from $B/\mfrak B = \ol A$, which is Artinian. In particular $B$ has finitely many prime ideals, all of which are maximal, and the Jacobson radical is equal to the nilradical, which is locally nilpotent. 
Applying \cite[\href{https://stacks.math.columbia.edu/tag/00JA}{Lemma 00JA}]{stacks-project} we get the second equality in \eqref{eq:A_first_product_locals}.

The above chain of canonical isomorphisms writes $A$ as a product of finitely many local rings, hence as the product of its finitely many localizations at maximal ideals. Moreover, these maximal ideals are the inverse images of the maximal ideals of $\ol A$, hence 
\begin{equation}\label{eq:A_product_locals}
A =  \prod_{m \in \mathrm{Specm} (\ol A)} A_m .
\end{equation}

We now bound the number of zeros of $I$. By Remark \ref{rem:pointsandHom}, they  are in bijection with $R$-algebra morphisms $ A \to R$, hence
\[
\# V_{R}\bigl(\{f_{1},\,f_{2},\,\dots,\,f_{r} \}\bigr) =  \# \,
\Hom(A,\,R)= \sum_{m \in \mathrm{Specm} (\ol A)}   \# \, \Hom(A_{m} , R) .  
\]
Notice that, if $\ol A/m$ is not (canonically) isomorphic to $k$ as a $k$-algebra, then $\#\Hom(A_{m} , R) = 0$: indeed each $R$-algebra morphism $\phi \colon A_m \to R$ reduces to a $k$-algebra morphism $ \phi {\otimes_R} k \colon \ol A_m  \to k$, which necessarily factors through a map $\ol A_m/m \to k$ which is injective since $\ol A_m/m$ is a field and surjective since it is a map of $k$-algebras.
Hence we have
\begin{equation}\label{eq:usata_1_volta}
\# V_{R}\bigl(\{ f_{1},\,f_{2},\,\dots,\,f_{r} \}\bigr) =  \sum_{m}   \#  \Hom(A_m , R) ,  
\end{equation}
where the sum is taken over the maximal ideals~$m \subset \ol A$ such that $\ol{A}/m=k$.

By Equation \eqref{eq:A_product_locals}, $A_m$ is a finite, torsion-free $R$-module, hence, by the same argument we used for $A$, it is a free $R$-module of finite rank. In particular, $A_m\otimes_{R}K$ is a finite $K$-algebra, hence Artinian, and by writing it as a product of local Artinian $K$-algebras, we see that the number of $K$-algebra morphisms $A_m\otimes_{R}K \to K$ is at most equal to
the $K$-dimension of $A_m\otimes_{R}K$. 
Since each morphism $A_m \to R$ induces a $K$-algebra morphism $A_m\otimes_{R}K \to K$, using that $A_m$ is free over $R$, we have the following inequalities
\[
\begin{aligned}
\# \, \Hom(A_{m}, R) & \le  \# \, \Hom(A_{m} \otimes_R K, K) 
\le  \dim_K (A_{m} \otimes_R K)  = \rank_R A_m =  
\\ & = \dim_k (A_m \otimes k) =\dim_k (\ol A_m) , 
\end{aligned}
\] 
which, together with \eqref{eq:usata_1_volta}, give the first inequality in Equation \eqref{eq:bound_zeroes}.
\end{proof}

\numberwithin{equation}{subsection}
\section{Computing saturation with approximated information}
\label{section saturation of an ideal}

In this section we restrict to the case where the valuation is discrete, that is, the maximal ideal $\mathfrak{m}$ 
is generated by a single uniformizer $\pi$.
Under this assumption, the rings $\mathcal{R}_n$ are Noetherian, by \cite[Chapter~0, Proposition~$7.5.4$]{MR3075000}.

Motivated by the results of the previous section, we investigate how to compute generators for the saturation $I^{\ec}$ of an ideal $I \subset \calR_n$. In the context of (convergent) power series, exact computations are obstructed by the infinite memory needed to represent general elements of $\calR_n$. A natural approach is to work with finite $\pi$-adic precision, thus representing a power series in $\calR_n$ as 
$f(x_1, \ldots, x_n) = P(x_1, \ldots, x_n) + O(\pi^N)$
with $P$ a polynomial with coefficients in $R$ or in a dense subset (e.g. $\Z$, which is dense in $\Z_p$). 
We are then interested in the following problem.

\begin{Prob}\label{prob:compute_saturation}
	Suppose that the valuation on $K$ is discrete, with uniformizer $\pi$.
	Given an ideal $I \subset \calR_n$, generated by power series $f_1, \ldots, f_r$, known with precision $O(\pi^N)$, and given a precision level $M \le N$, compute, if possible, generators for $I^{\ec}$ with precision $O(\pi^M)$.
\end{Prob}

It is not always possible to solve the above problem. Even when $M=1 \ll N$, there are cases where the information modulo~$\pi^N$ is not enough to determine the reduction of $I^{\ec}$. Consider, for example, $I = (f_1, f_2) \subset \calR_2$ with both
$f_1$ and $f_2$ being equal to $x + O(\pi^N)$: this approximation might arise from $f_1 = f_2=x$ or from $f_1 = x$ and $f_2 = x + \pi^{a} y$ for some integer $a\ge N$; in the first case we would have $\ol{I^{\ec}} = (x)$, while in the second case $\ol{I^{\ec}} = (x,y)$. 

Problem~\ref{prob:compute_saturation} is easy if $I$ is generated by only one power series. Suppose $I = (f_1)$, with $f_1 = p_1(x_1, \ldots, x_n)+O(\pi^N)$ for some polynomial $p_1$, and let $m$ be the minimum valuation among the coefficients of $p_1$. Then $I^{\ec}$ is determined modulo $\pi^M$ if and only if $M\le N-m$, and in that case $I^{\ec} = (f_1/\pi^m) = (p_1/\pi^m + \calO(\pi^M))$. 

We are especially interested in Problem~\ref{prob:compute_saturation} when $M=1$, since it allows us to apply Theorem~\ref{Theorem main}, and to check whether $I$ has finitely many zeros in $(\calO_{K^c})^n$, with $K^{c}$ the completion of an algebraic closure of $K$, and potentially bound the number of zeros using Theorem \ref{main theorem section introduction}.  
This is also desirable in arithmetic applications, such as Chabauty's and Skolem's methods, where one bounds the solutions of a Diophantine problem with the zeros of power series, which are usually computed with finite precision. 

\begin{Remark}
	The saturation of an ideal $I$ may be interpreted geometrically as a Zariski closure. Indeed, if $X$ is the $R$-scheme $\Spec(\Rcal_{n}/I)$, then its generic fiber is $X_K = \Spec \bigl((\Rcal_{n}/I) \otimes K\bigr) = \Spec (\Tcal_{n}/I^{\mathrm{e}}) $, and  the scheme $Z = \Spec(\Rcal_{n}/I^{\ec})$ is supported on the Zariski closure of $X_K$ inside $X$, fitting into the following commutative diagram
	\begin{center}
		\begin{tikzcd}
			X_K \arrow{rr} \arrow{d} && Z = \ol{X_K}  \arrow[r] \arrow{dr} & X \arrow{d} 
			\\
			\Spec(K) \arrow{rrr} & & & \Spec(R) 
		\end{tikzcd}
	\end{center}
\end{Remark}

\begin{Remark}[Relation with Gr\"obner bases] \label{rmk:saturation_vs_Grobner}
Problem~\ref{prob:compute_saturation} is also relevant in the context of Gr\"obner bases of ideals $J$ of $\calT_n$, as defined in \cite[Definition~$2.16$]{MR4007445}.

Suppose $J$ is an ideal of $\calT_n$ generated by power series  $f_1, \ldots, f_r \in \calR_n$ and let $I$ be the ideal of $\calR_n$ generated by the same power series.
By Proposition~$2.25$ in loc.cit., a Gr\"obner basis of $J$, up to multiplying by scalars, is also a set of generators of $I^{\ec}$.
Conversely, given generators $g_i$'s for the saturation of $I$, by Propositions $2.25$ and $2.28$ in \cite{MR4007445}, we can find a Gr\"obner basis of $J$ by first computing a Gr\"obner basis 
$\ol{h_j} = \sum_i \ol{\mu_{ij}} \cdot \ol {g_i}$  for $\ol{I^{\ec}}$ in the polynomial ring $k[x_1, \ldots, x_n]$ and then lifting them to $h_j = \sum_i \mu_{ij} g_i$.

These equivalences also hold when working with finite precision and indeed, as noticed in \cite{MR4007445} after the proof of Theorem~$3.8$, 
it is not always possible to determine (the reduction of) a Gr\"obner basis starting from approximate information in $\calT_n$.
\end{Remark}

\subsection{Saturation and syzygies}

We start by looking at  Problem~\ref{prob:compute_saturation} more ``theoretically'', i.e.,  working with exact power series. 
By definition, the saturation $I^{\ec}$ contains those power series $g$ such that $\pi^l g$ lies in $I$, and more precisely in $I \cap (\pi^l)$, for some $l\in \Z_{\ge 0}$.
In other words, if we define the ideals
\begin{equation}
\label{equation ideals Icerchietto n}
I_{l}\coloneqq  \frac{1}{\pi^{l}} (I\cap  \pi^{l} \calR_n )  = \left( I :  \pi^{l} \calR_n \right) ,
\end{equation}
we have
\begin{equation}
\label{equation catena Il}
I=I_{0} \subseteq I_{1} \subseteq I_{2} \subseteq I_{3} \subseteq \cdots   \subseteq \bigcup_{l\in\N} I_{l} = I^{\ec} .
\end{equation}
Since $\calR_n$ is Noetherian, the chain of ideals $\{I_{l}\}_{l}$  stabilizes after a finite number of steps. Indeed, it stabilizes at the first index $m$ such that  $I_m = I_{m+1}$ since
we have the recursion
\begin{equation}\label{eq:one_step_saturation}
I_{l+1} =   \frac{1}{\pi} (I_l \cap \pi \calR_n )  = \left( I_l : \pi \calR_n  \right) .
\end{equation}

We can recursively ``find'' generators for the ideals $I_l$ using syzygies. 
Recall that, given a  finitely generated ideal $J$ of a ring $A$ with a set of generators  
$\mathcal{G}=\{f_{1},\dots,f_r\}$, 
we have the (first) module of syzygies 
\begin{equation*}
\Syg_{\calG, \,J}
 := \big\{(a_{1},\ldots,a_{r}) \in A^r : \, a_{1}f_{1}+  \cdots+a_{r}f_{r} = 0  \big\} .
\end{equation*}
Fix generators $f_1, \ldots, f_r$ for $I_l$. If a power series $g$ lies in  $I_{l+1}$,
then we can write $\pi g = a_1 f_1 + \ldots + a_rf_r$ for suitable $a_i \in \calR_n$, which, after reducing modulo~$\pi$, implies
that $(\ol {a_1}, \ldots, \ol{a_r})$ is a syzygy for $\ol{I_l} \subset k[x_1, \ldots, x_n]$, with generators $\ol \calG =\{ \ol{f_1}, \ldots , \ol{f_r} \}$. 
Conversely,
for any syzygy 
$s=(a_{1},\ldots,a_{r}) \in \Syg_{\ol\calG, \ol{I_l}}$
of the reduction $\ol{I_l}$,
and for any lift $(\tilde a_{1},\ldots, \tilde a_{r})$  of $s$ in $(\calR_n)^r$, we can define the power series 
\begin{equation}
\label{eq:generators_gs}
g_{s} \coloneqq \frac 1 \pi \sum_{i=1}^{r}\tilde a_{i} f_{i} ,
\end{equation}
which belongs to $I_{l+1}$. Notice that the definition of $g_s$ also depends on the choice of the lifts $\tilde a_i$, and that $g_s$ is ``new'', i.e.,  $g_s \notin I_l$, if and only if $s$ is not the reduction of a syzygy in $\Syg_{\calG,{I_l}}$. 
The following lemma implies that, after adding sufficiently many such elements $g_s$ to $I_l$, we obtain $I_{l+1}$; this is made precise in Corollary~\ref{cor:syzygies_construction}.

\begin{Lemma} \label{lem:saturation}
Suppose that the valuation on $K$ is discrete with uniformizer $\pi$. Let $I$ be an ideal of $\calR_n$ with generators $\mathcal{G}=\{f_{1}, \dots,f_{r}\}$. Let $\overline{I} \subset k[x_{1},\,\dots,\,x_{n}]$ be the reduction of $I$, with generators 
$\overline{\mathcal{G}}=\{\overline{f_{1}},\dots, \overline{f_{r}}\}$.
Consider the syzygy modules $\Syg_{\calG,I}$ and $\Syg_{\ol{\calG},\overline{I}}$, and also the reduction 
\begin{equation*}
\overline{\Syg_{\calG,\,I}}\coloneqq \big\{(\overline{a_{1}},\dots,\overline{a_{r}}) \,\colon\, (a_{1},\ldots,a_{r})\in\Syg_{\calG,\,I} \big\} .
\end{equation*}
Then we have: 
\begin{enumerate}[a)]
\item\label{item:is_saturated}
the reduction $\overline{\Syg_{\calG,I}}$ is a subset of $\Syg_{\ol\calG, \ol I}$  and equality holds if and only if $I=I^{\ec}$;
\item \label{item:finite_quotient_sat} 
the module $\Syg_{\ol{\calG},\overline{I}}$ and its quotient $\Syg_{\ol{\calG},\overline{I}} / \ol{\Syg_{\calG,I}}$ are finitely generated $k[x_1, \ldots, x_n]$-modules; 
\item \label{item:generation_saturation} given a set of generators $\{ s_1, \ldots, s_t \}$ of the quotient $\Syg_{\ol{\calG},\overline{I}} / \ol{\Syg_{\calG,I}}$, for any choice of functions $g_{s_1}, \ldots, g_{s_t}$ as in \eqref{eq:generators_gs} we have
\[
\frac{1}{\pi} (I \cap \pi \calR_n) =  (f_1, \ldots, f_r, g_{s_1}, \ldots, g_{s_t}) .
\]	
\end{enumerate}
\end{Lemma}
\begin{proof}
In part \ref{item:is_saturated} 
one implication follows from part~\ref{item:generation_saturation}: if the quotient $\Syg_{\ol{\calG},\overline{I}} / \ol{\Syg_{\calG,I}}$ is trivial, then $I = I_1$, and by induction, using \eqref{eq:one_step_saturation}, the whole chain of ideals $I_0 \subseteq I_1 \subseteq \ldots$ in \eqref{equation ideals Icerchietto n} is constant and equal to $I$. 
Conversely, if $\Syg_{\ol{\calG},\overline{I}} / \ol{\Syg_{\calG,I}}$ is not trivial, then we can pick a non-zero syzygy $s$ in the quotient, so that the series $g_s$ defined in \eqref{eq:generators_gs} belong to $I^{\ec}$ but not to $I$. 

Part~\ref{item:finite_quotient_sat} is immediate since  $\Syg_{\ol \calG, \ol I} \subseteq (k[x_1, \ldots, x_n])^r $ is finitely generated by Noetherianity of $k[x_1, \ldots, x_n]$, and a fortiori any quotient of $\Syg_{\ol \calG, \ol I}$ is finitely generated.

One inclusion in part~\ref{item:generation_saturation} is clear, since all $f_i$ and $g_{s_i}$ belong to $\tfrac{1}{\pi}(I \cap \pi \mathcal{R}_n)$. 
For the other inclusion take an element $f$ of  $\tfrac 1\pi (I \cap \pi \calR_n)$, which can be written as
$$
f= \frac 1 \pi  (\lambda_{1}f_{1} + \ldots + \lambda_r f_r) ,
$$
for suitable power series $\lambda_i$. 
Since $(\overline{\lambda_{1}},\, \dots,\, \overline{\lambda_{r}})$ is a syzygy for $\ol I$, it 
is equal, in the quotient 
$\Syg_{\ol{\calG},\overline{I}}/\ol{\Syg_{\calG,I}}$,  to a combination 
$\ol{\mu_1} s_1 + \ldots + \ol{\mu_t} s_t$    
for suitable  $\ol{\mu_i} \in k[x_1, \ldots, x_n]$. 
Choosing lifts $\mu_i \in \calR_n$ of $\ol{\mu_i}$ and using the definition \eqref{eq:generators_gs} of $g_{s_i}$ (which also depends on the choice of lifts)
we can write    
$$
f - \mu_1 g_{s_1} - \ldots - \mu_t g_{s_t} = \frac{1}{\pi} (\nu_{1}f_{1} + \ldots + \nu_r f_r)
$$
for certain power series $\nu_i$ satisfying the following congruence
\[
(\ol{\nu_1}, \ldots, \ol{\nu_r}) = (\overline{\lambda_{1}},\, \dots,\, \overline{\lambda_{r}}) - \ol{\mu_1} s_1 - \ldots  -\ol{\mu_t} s_t    .
\]
By the definition of the $\ol{\mu_j}$'s, the above syzygy is trivial in the quotient 
$\Syg_{\ol{\calG},\overline{I}} / \ol{\Syg_{\calG,I}}$, hence it can be lifted to  a syzygy $(\nu_1', \ldots, \nu_r')$ of $I$. In other words, we have $\sum \nu'_i f_i = 0$ and $\nu_i \equiv \nu_i' \pmod \pi$. Writing $\nu_i = \nu_i' + \pi \omega_i$ for suitable $\omega_i\in \calR_n$, we deduce that 
\[
f - \sum_{i=1}^t\mu_i g_{s_i}  =  \frac{1}{\pi} \bigl( (\nu_{1}' + \pi \omega_1)f_{1} + \ldots + (\nu_{r}' + \pi \omega_r) f_r \bigr) =   \omega_1 f_{1} + \ldots + \omega_r f_r ,
\]
hence $f$ belongs to $(f_1, \ldots, f_r, g_{s_1}, \ldots, g_{s_t})$, which proves the desired inclusion.
\end{proof}

\begin{Corollary}\label{cor:syzygies_construction}
Suppose that the valuation on $K$ is discrete. Let $I$ be an ideal of $\calR_n$, and consider the ideals $I = I_0 \subseteq I_1  \subseteq \ldots$ defined in \eqref{equation ideals Icerchietto n}. 
Then, for each  $l \in \Z_{\ge 0}$, for each choice of generators $\calG$ for $I_l$ and generators $\{s_{1},\dots,s_{t}\}$ for 
$\Syg_{\ol{\calG},\overline{I_{l}}} / \ol{\Syg_{\calG,I_{l}}}$, we have 
\begin{equation*}
I_{l+1} = \langle I_l,\,g_{s_{1}} , \dots,\,g_{s_{t}}\rangle ,
\end{equation*}
where the $g_{s_i}$'s are defined as in Equation~\eqref{eq:generators_gs}. In particular, if $\Syg_{\ol{\calG},\overline{I_{l}}} = \ol{\Syg_{\calG,I_{l}}}$, then $I_l = I^{\ec}$.
\end{Corollary}

Corollary~\ref{cor:syzygies_construction} gives a constructive way to compute the ideals $I_l$. Notice that one can simply choose the elements $s_i$ 
to be generators of $\Syg_{\ol{\cal G}, \ol{I_l}}$, which only requires computing syzygies over a polynomial ring.
Moreover, the functions $g_{s_{1}}, \dots, g_{s_{t}}$, defined in \eqref{eq:generators_gs}, are combinations of the $f_{i}$ divided by $\pi$, hence if we know the $f_{i}$'s with $\pi$-adic precision $N$, then we can compute the $g_{s_j}$'s with $\pi$-adic precision $N-1$. This gives the following algorithm. An implementation is attached to the Arxiv version of this article.

\bigskip
\begin{algorithm}[H]
\caption{Computing generators for $I_{1}, \ldots, I_{N-1}$ with decreasing $\pi$-adic precision}
\label{alg: generators Il}

\SetKwInOut{Input}{Input}\SetKwInOut{Output}{Output}
\SetKwFor{ForEach}{for each}{do}{end}

	\Input{A set $\mathcal{G}_0 = \bigl\{q_1 + O(\pi^N), \ldots, q_r+ O(\pi^N) \bigr\}$ of generators of $I \subset \calR_n$.} 
\Output{A set of generators $\calG_l$ of $I_{l}$ with precision $O(\pi^{N-l})$ 
for each $l = 1, \ldots, N-1$.}

\BlankLine

\ForEach{$l \in \{1, \dots, N-1\}$}{
Denote $\calG_{l-1} = \bigl\{p_1+ O(\pi^{N-l+1}), \ldots, p_u+ O(\pi^{N-l+1}) \bigr\}$\;
Compute generators $s_{1}, \dots, s_{t} \in \bigl(k[x_1, \ldots, x_n]\bigr)^u$ of $\Syg_{\{ \ol{p_1}, \ldots, \ol{p_u}\},\overline{I_{l-1}}}$  (see \cite{MR3424044})\;
\ForEach{$j \in \{1, \dots, t\}$}{
Choose a lift $\tilde{s_{j}} = (\tilde a_1, \ldots, \tilde a_u) \in \Rcal_{n}^{u}$ of $s_j $ and set $g_{{s}_{j}} := \tfrac 1 \pi \sum_{i=1}^u \tilde a_i p_i$ as in \eqref{eq:generators_gs}\;
}
Set $\calG_l=\calG_{l-1}   \cup  \bigl\{g_{{s_{1}}} + O(\pi^{N-l}), \dots, g_{{s_{t}}} + O(\pi^{N-l})\bigr\}$
}

\Return{$\calG_1, \ldots , \calG_{N-1}$\;}
\end{algorithm}

\smallskip

\begin{Remark}\label{rmk:less_generators} 
    Some of the generators computed by the above algorithm might be redundant, and we might be able to prove it even working with finite precision. 
    For example, if the first element $g $ of $\calG_{l-1}$ is a multiple of $\pi$, then $s_1 = (1,0,\ldots, 0)$ is a syzygy in $\Syg_{\ol{\calG_{l-1}},\overline{I_{l-1}}}$, and if we take  $\tilde s_1 = (1,0,\ldots, 0)$ then $\calG_l$ contains both an approximation of $g$ and an approximation $g_{s_1}$ of $\frac{1}{\pi} g$, but the first can be removed from $\calG_l$. 
    We can generalize this observation as follows.

    To avoid confusion, given an approximation $g$ of a power series in $I_l$, we denote by $\widehat g$ its exact counterpart. 
    For example, the elements $q_1 + O(
    \pi^{N}), \ldots, q_r+O(\pi^N)$ in $\calG = \calG_0$ are approximations of certain $\widehat{q_i}$. Then, if we choose the lifts $\tilde a_i$ in the algorithm to be polynomials in $A:=R[x_1, \ldots, x_n]$, we can  recursively write the elements $\widehat{p_j}$ of $\calG_l$ as an (exact) combination of the $\widehat{q_i}$'s. In other words, we can compute $b_{ij} \in A$ such that 
    \[
    \widehat{p_j} = \frac{1}{\pi^l} \sum_{i=1}^r b_{ij} \cdot \widehat{q_i} .
    \]
    We then look at the $A$-module $M  \subset A^r$ generated by the vectors $B_j = (b_{1j}, \ldots, b_{rj})$. If we can generate $M$ with fewer generators $B'_j = (b'_{1j}, \ldots, b'_{rj})$,  
    then we are also able to generate $I_l$ with fewer than $\# \calG_l$ power series: a smaller set of generators is given by the series
    \[
\widehat{p_j'} := \frac{1}{\pi^l} \sum_{i=1}^r b'_{ij} \widehat{q_i} 
\quad \text{ approximated as }  \quad 
p_j' =\frac{1}{\pi^l} \sum_{i=1}^r b'_{ij} q_i + O(\pi^{N-l}) . 
	\]
	For example, in the case $R = \Z_p$, we can approximate convergent power series with polynomials in $\Z[x_1, \ldots, x_n]$ and in practice we can check if some of the $B_j$'s are redundant using the \emph{mingens} function in Macaulay2 \cite{M2}.  
\end{Remark}

\subsection{Some dimensionality results}
The above Algorithm~\ref{alg: generators Il} would be sufficient to solve  Problem~\ref{prob:compute_saturation} if we knew an index $l\le N-M$ such that $I_{l} = I_{l+1} = I^{\ec}$. 
We are not aware of any general algorithm to determine whether this happens using approximate information. 
For example, since we do not have access to the generators of $I_l$ with infinite precision, it is in general not possible to compute $\Syg_{\calG_l,I_{l}}$ and to  check the termination condition $\Syg_{\ol{\calG_l},\overline{I_{l}}} = \ol{\Syg_{\calG_l,I_{l}}}$ in Lemma \ref{lem:saturation}.
Sometimes, though, it happens that $\Syg_{\ol \calG_l, \ol I_{l}}$ is as small as possible, and this gives a sufficient criterion  for having $I_l = I^{\ec}$, as in the following proposition.

\begin{Proposition}\label{prop:approxiation_sometimes_ok}
    Suppose that the valuation on $K$ is discrete and let $I$ be an ideal of $\calR_n$ generated by $r$ power series.
    
    If for some integer $l \ge 0$ we have  $\dim\bigl(k[x_1, \ldots, x_n]/\ol{I_l}\bigr) < n-r$, then $I^{\ec} = \calR_n$.
    If for some integer $l \ge 0$ we have $\dim\bigl(k[x_1, \ldots, x_n]/\ol{I_l}\bigr) = n-r$ and $I_l$ can be generated by $r$ power series, then $I_l = I^{\ec}$. 
    In particular, if $\dim\bigl(k[x_1, \ldots, x_n]/ \ol{I}\bigr) = n-r$, then $I = I^{\ec}$. 
\end{Proposition}
\begin{proof}
We first analyze the case $\dim\bigl(k[x_1, \ldots, x_n]/\ol{I_l}\bigr) < n-r$. 
    The ideal $I^{\mathrm{e}}$ is generated by the same $r$ generators of $I$, hence the quotient $\calT_n/I^{\mathrm{e}}$, when it is not the trivial ring, has dimension at least $n-r$ by  Krull's  Hauptidealsatz, together with the fact that $\calT_n$ is regular (see \cite[Proposition~17, Chapter~2]{MR3309387}), hence catenary. If, by way of contradiction, $I^{\ec}$ is not the full $\calR_n$, then the quotient $\calT_n/I^{\mathrm{e}}$ is not the trivial ring  and, using Theorem \ref{Theorem main}, we have
 \[
 \dim\left( k[x_1, \ldots, x_n]/\ol{I^{\ec}} \right) = \dim\left(\calT_n/I^{\mathrm{e}} \right) \ge n-r .
 \]
 Anyway the ring $k[x_1, \ldots, x_n]/\ol{I^{\ec}}$ is a quotient of $k[x_1, \ldots, x_n]/\ol{I_l}$, hence it has dimension at most  $\dim\bigl(k[x_1, \ldots, x_n]/\ol{I_l}\bigr)$ which is smaller than $n-r$, a contradiction.

We now turn to the last part of the statement. 
It is enough to treat the case $l=0$, i.e.,   $I_l = I$. 
Moreover, it is enough to prove that, for some extension $L \supset K$, with ring of integers $\calO_L$, the ideal  $I \cdot \calO_L$ is saturated in $\calO_L\langle x_1, \ldots, x_n \rangle$.  
 Indeed, using that $\calO_L$ is a free $R$-module of finite rank, we can deduce that $(I \cdot \calO_L) \cap \calR_{n} = I$, hence that $I$ is also saturated.
We now look for an extension $L$ such that $I \cdot \calO_L$ is generated by power series $f_1, \ldots, f_r$ whose reductions modulo~$\mfrak$ form a regular sequence and, using this hypothesis, we prove that $I \cdot \calO_L$ is saturated.

Let $\{ g_{1}, \ldots, g_{r} \}$ be a generating set for $I$, and let $k^{s}$ be the separable closure of $k$. First, we prove that the extension $\ol{I}_{k^{s}}$ of $\ol{I}$ to $k^{s}[ x_{1}, \ldots, x_{n} ]$ can be generated by a regular sequence  $( \ol{f_{1}}, \ldots, \ol{f_{r}} )$ 
which, up to reordering the $g_j$'s, is of the form
\begin{equation}\label{eq:reduced_regular_seq}
\ol{f_{i}} = \lambda_{i, \, i} \cdot \ol{g_{i}} + \lambda_{i, \, i+1} \cdot \ol{g_{i+1}} + \ldots + \lambda_{i, \, r} \cdot \ol{g_{r}} , 
\end{equation}
for suitable constants $\lambda_{i, \, j} \in k^{s}$ such that $\lambda_{i,i} \neq 0$. 

Since $k^{s}[x_{1}, \dots, x_{n}]$ is a domain, we set $\overline{f_{1}} = \overline{g_{1}}$ which is different from zero, otherwise $\dim\bigl(k[x_1, \ldots, x_n]/\overline{I}\bigr)$ would exceed $n-r$. Proceeding by induction, suppose we have constructed a regular sequence $(\overline{f_{1}}, \dots, \overline{f_{s}})$ as in \eqref{eq:reduced_regular_seq} for some $s < r$. Let $C_{1}, \dots, C_{a}$ be the irreducible components of the subvariety defined by these polynomials. For any polynomial $\ol f$,  the sequence $(\overline{f_{1}}, \dots, \overline{f_{s}}, \overline{f})$ is regular if and only if $\overline{f}$ does not vanish identically on any $C_i$. Since $\dim\bigl(k^s[x_1, \ldots, x_n]/\overline{I_{k^s}}\bigr) = n-r < n-s$, the polynomials $\overline{g_{s+1}}, \dots, \overline{g_{r}}$ cannot simultaneously vanish on any $C_i$. As $k^s$ is infinite, there exists a linear combination of them with coefficients in $k^{s}$ which does not vanish on all $C_{i}$'s. Up to reordering, we may assume the first coefficient of this combination to be non-zero and define $\overline{f_{s+1}}$ as this combination.

Once we have constructed the $\ol{f_i}$ as in \eqref{eq:reduced_regular_seq}, we pick $L$ to be a finite extension of $K$ whose residue field $k_L$ contains all the $\lambda_{i,j}$'s. Then, we can lift the polynomials $\ol{f_1}, \ldots, \ol{f_r}$ to generators $f_1, \ldots, f_{r}$ of $I \cdot \calO_L$ simply by 
choosing lifts $\tilde \lambda_{i,j} \in \calO_L$ of $\lambda_{i,j}$ and defining 
\begin{equation*}
f_{i} \coloneqq \tilde{\lambda}_{i, \, i} \cdot g_{i} + \tilde{\lambda}_{i, \, i+1} \cdot g_{i+1} + \ldots + \tilde{\lambda}_{i, \, r} \cdot g_{r}. \end{equation*}
Notice that the $f_i$'s generate $I\cdot \calO_L$ since the $g_i$'s are generators and the matrix $(\tilde \lambda_{i,j})$ has non-zero determinant modulo
the maximal ideal of $\calO_L$.

Finally we prove that $I \cdot \calO_L$ is saturated, using syzygies. Since $\{ \ol{f_{1}}, \ldots, \ol{f_{r}} \}$ is a regular sequence in $k_{L}[x_{1}, \ldots, x_{n}]$, \cite[Corollary~$17.5$]{eisenbud2013commutative} implies that the associated Koszul complex is exact, hence the $(r - 1)$-th cohomology group of the Koszul complex vanishes (the Koszul complex of Eisenbud is slightly different from the usual one). It follows that the syzygy module of $\{ \ol{f_{1}}, \ldots, \ol{f_{r}} \}$ is generated by the elements
\begin{equation*}
( 0, \ldots, 0, -\ol{f_{j}}, 0, \ldots, 0, \ol{f_{i}}, 0, \ldots, 0 ) , \qquad\qquad 1 \le i < j \le r ,
\end{equation*}
where $-\ol{f_{j}}$ is in position $i$ and $\ol{f_{i}}$ is in position $j$. Since these elements are the reductions of the elements
\begin{equation*}
( 0, \ldots, 0, -f_{j}, 0, \ldots, 0, f_{i}, 0, \ldots, 0 ) , \qquad\qquad 1 \le i < j \le r ,
\end{equation*}
in the syzygy module of $\{ f_{1}, \ldots, f_{r} \}$, Lemma~\ref{lem:saturation} part~\ref{item:is_saturated} implies that $I \cdot \calO_L$ is saturated, as desired.
\end{proof}

The above statement gives a sufficient criterion to answer Problem~\ref{prob:compute_saturation} using Algorithm~\ref{alg: generators Il}, together with Remark~\ref{rmk:less_generators}. For example, if $I$ is generated by $r$ elements and $k[x_1, \ldots, x_n]/\ol I$ has dimension $n-r$, then $I^{\ec} = I$, which we know with precision $O(\pi^N)$. A slightly less trivial example is $I =(f_1, f_2)$ for some $f_1, f_2$ known with precision $N \ge 2$, and $f_2 = gf_1 + \pi f_3 + O(\pi^N)$ for some  $f_3$ such that $ k[x_1, \ldots, x_n]/(\ol{f_1}, \ol{f_3})$ has dimension $n-2$. The algorithm gives $I_1 =(f_1, f_2, f_3' := \frac1\pi(f_2-gf_1))$ 
and,  by Remark \ref{rmk:less_generators}, $I_1$ is also generated by $f_1, f_3'$; in particular $\ol{I_1} = (\ol{f_1}, \ol{f_3'}) = (\ol{f_1}, \ol{f_3})$, hence, by Proposition~\ref{prop:approxiation_sometimes_ok}, we have $I^{\ec} = I_1$, which we know with precision $O(\pi^{N-1})$.

We now look at an example where Proposition \ref{prop:approxiation_sometimes_ok} cannot be applied, even though, starting from approximate information, we are able to compute the reduction of the saturation. 

\begin{Example}
\label{example saturation and dimension}
Consider an ideal $I =  (f_1, f_2) \subset \Z_{p}\langle x, y \rangle$ with $f_1 = x^{2} + py + O (p^2)$ and $f_2 =   xy +O( p^2)$. In other words there exist power series $g_1, g_2 \in \Z_{p}\langle x,  y\rangle$ such that $f_1 = x^{2} + py + p^2 g_1$ and $f_2 =   xy + p^2 g_2$. Then the 
syzygy module associated to $\ol{I} = ( x^{2}, xy )$ is generated by $s = ( y, -x )$, and consequently  
\begin{equation*}
I_{1} 
= \left( x^{2} + py + p^2 g_1, xy + p^2 g_2, y^{2} + pyg_1 -pxg_2 \right),
\quad\text{hence}\quad 
 \ol{I_{1}} = ( x^{2},  xy,  y^{2} ) .
\end{equation*}
The syzygy module associated to $\ol{I_{1}}$ is generated by the syzygies 
\begin{equation*}
( y,  -x,  0) , \quad ( y^{2}, 0,  -x^{2})  , \quad ( 0,  y,  -x) ,
\end{equation*}
which all lift to syzygies of $I_1$, namely
\[
(y,-x,-p),\quad
(y^2-pxg_2,pxg_1,-x^2-py),\quad
(-pg_2,y+pg_1,-x)
\]



Then, by Corollary \ref{cor:syzygies_construction}, we have $I_{1} = I^{\ec}$. Notice that $\ol{I_{1}}$ cannot be generated by $2$ elements, but the condition $\dim \bigl( \F_{p}[x, y] / \ol{I_{1}} \bigr) = 0 = n-r$ holds.
\end{Example}

Other examples of this type suggest that Proposition~\ref{prop:approxiation_sometimes_ok} might hold in greater generality. This leads us to formulate the following question:

\begin{Question} \label{conjecture}
Suppose that the valuation on $K$ is discrete. 
Is it true that for all ideals $I \subset \calR_n$ generated by $r$ power series, if for some  $l \in \Z_{\ge 0}$ we have $\dim\bigl(k[x_1, \ldots, x_n]/ \ol{I_{l}}\bigr) = n - r$, then $I_{l} = I^{\ec}$?
\end{Question}

\subsection{Saturation and ideals generated by polynomials...}
We look at another approach, inspired by Example~\ref{example saturation and dimension}. There, we were able to prove that $I_1 = I^{\ec}$ starting from approximate power series simply by giving the name ``$g_i$'' to the information hidden by the approximation. 
This suggests 
another possible approach to Problem~\ref{prob:compute_saturation}, namely to treat the information we don't know as a new variable: if we know the power series $f_i$ with precision $\pi^N$, say $f_i = p_i + O(\pi^N)$ for certain polynomials $p_i(x_1, \ldots, x_n)$, we can look at  $f_i'(x_1, \ldots, x_n, y_i) = p_i(x_1, \ldots, x_n) + \pi^N y_i$ and try to compute the saturation of the ideal generated by the (exact) power series $f_i'$. By the following lemma, if the approximated  result does not depend on the variables $y_{j}$'s, then the same result also approximates the saturation of $I$.

\begin{Lemma}
    \label{lem: approximation is useful}
    Suppose that the valuation on $K$ is discrete, with uniformizer $\pi$, and let $N$ be a positive integer. Let $I = (f_1, \, \ldots, \, f_r)$ be an ideal of $\calR_n$ and let $p_i$ be polynomials in $R[ x_{1}, \, \ldots, \, x_{n}]$ such that 
    $f_i \equiv p_i \pmod {\pi^N}$. Consider the ideal $J = ( p_{1} + \pi^{N}y_{1}, \ldots,  p_{r} + \pi^{N}y_{r} )$ of $R\langle x_1, \, \ldots, \, x_n, \, y_1, \, \ldots, \, y_r\rangle$. 
    
    For every positive integer $M$, 
    if $J^{\ec}$ can be generated by power series $h_1, \ldots, h_s$ of the form $h_i = q_i(x_1, \ldots, x_n) + O(\pi^M)$, that is, the polynomials $(h_i \bmod \pi^M)$ do not contain the variables~$y_j$, then $I^{\ec}$ can be generated by power series $g_1, \ldots, g_s \in \calR_n$ of the form $g_i = q_i(x_1, \ldots, x_n) + O(\pi^M)$. 
\end{Lemma}
\begin{proof}
Denoting by $S_1, \dots, S_r \in \calR_n$  the power series such that 
$$f_i = p_i + \pi^N S_i , $$
we prove that $I^{\ec}$ is generated by the power series
\begin{equation*}
g_{j}(x_{1}, \, \ldots, \, x_{n}) := h_j\bigl( x_{1}, \, \ldots, \, x_{n}, \, S_{1}(x_{1}, \, \ldots, \, x_{n}), \, \ldots, \, S_{r}(x_{1}, \, \ldots, \, x_{n}) \bigr) , \qquad 1 \le j \le s ,
\end{equation*}
which implies our claim. 
To prove it, it is enough to show that  the ideal  $L \subset \calR_{n}$ generated by $g_{1}, \, \ldots, \, g_{s}$ is saturated and satisfies $I \subseteq L \subseteq I^{\ec}$. 

The latter condition is a specialization of the analogous containments $J \subseteq (h_1,\ldots, h_s ) \subseteq J^{\ec}$. 
For example, since $h_j \in J^{\ec}$, there exist power series $c_{1}, \ldots, c_{r}$ in $R\langle x_1, \ldots, x_n, y_1, \ldots, y_r\rangle$ and an integer $m$ such that
\begin{equation*}
\pi^{m}h_j = \sum_{i = 1}^{r}c_{i} \cdot \bigl( p_{i}(x_{1}, \ldots, x_{n}) + \pi^N y_{i} \bigr) .
\end{equation*}
Specializing $y_{i}$ to $S_{i}$, we see that
$\pi^{m}g_{j} = \sum_{i = 1}^{r}\tilde c_{i}' \cdot f_{i}$, with $c'_i\in \calR_n$ the specializations of $c_i$, hence that $g_{j} \in I^{\ec}$ for every $1 \le j \le s$.

Now the saturation of $L$. By Lemma~\ref{lem:saturation} part~\ref{item:is_saturated} it is enough to show that all the syzygies of $\{\ol{g_1}, \ldots, \ol{g_s}\}$ can be lifted to syzygies of $\{g_1, \ldots, g_s\}$, where $\ol{\cdot}$ denotes the reduction modulo~$\pi$ as usual. This is again a specialization argument. Indeed, since  
$h_j \bmod \pi^M$ does not involve the variables $y_{h}$, then we have 
\begin{equation*}
\ol{g_{j}} = \ol {{h}_{j}} = \ol{q_j}  . 
\end{equation*}
It follows that any syzygy $\ol c = (\ol{c_{1}}, \ldots, \ol{c_{s}})$ of the generators  $\{\ol{g_1}, \ldots, \ol{g_s}\}$  of $\ol L$ is also a syzygy of the  generators $\{\ol{h_1}, \ldots, \ol{h_s}\}$ of $\overline{J^{\ec}}$. Since $J^{\ec}$ is saturated, Lemma~\ref{lem:saturation} part~\ref{item:is_saturated} implies that there exists a syzygy $c = (c_{1}, \ldots, c_{s})$ for $\{h_1, \ldots, h_s\}$ in $R\langle x_1, \ldots, x_n, y_1, \ldots, y_r\rangle$ whose reduction is $\ol c$. In particular, from $c$ we can obtain a lift $c' = (c_{1}', \ldots, c_{s}')$ of $\ol c$ in $\Syg_{\{g_1, \ldots, g_s\},L}$ by specializing the $y_i\mapsto S_i$ (notice that $\ol{c_j}$ does not contain the variables $y_h$, hence this specialization does not change the reduction)
\end{proof}

The primary advantage of $J$ over $I$, in Lemma~\ref{lem: approximation is useful}, is that we know exact generators of it and we can compute its saturation by passing to the polynomial world. 
Denote by $\tilde{J}$ the ideal of $R[x_{1},  \ldots,  x_{n}, y_{1},  \ldots,  y_{r}]$ generated by the  polynomials $p_i + \pi^N y_i$, (which are also generators for $J$). Using \cite[Algorithm~$18$ and Theorem~$19$]{MR3295981} we can compute generators for its saturation in the polynomial ring
\begin{equation}    \label{eq:tilde J ec}
\tilde{J}^{\ec} :=  \left\{ f \in R[x_{1},  \ldots,  x_{n}, y_{1},  \ldots,  y_{r}] :  a f \in \tilde J \text{ for some } a \in R \setminus \{0\} \right\} .
\end{equation}
The next proposition shows that the generators of $\tilde{J}^{\ec}$ generate $J^{\ec}$ too.

\begin{Proposition} \label{pro: aux 1; lem: approximation is useful}
    Suppose the valuation on $K$ is discrete and let $I$ be an ideal of $R[x_1, \ldots, x_n]$. 
    If $I$ is saturated, i.e.,  $I = \bigl(I \cdot K[x_1, \ldots, x_n]\bigr) \cap R[x_1, \ldots, x_n]$, then its extension $I \cdot \calR_n$ is saturated in $\calR_n$.
\end{Proposition}
\begin{proof}
    Once we choose a set $\calG$ of generators for $I$, the statement follows from Lemma \ref{lem:saturation}. Indeed $I \cdot \calR_n$ is also generated by $\calG$ and, denoting by $\ol\calG$ the set  of reductions of the elements of $\calG$, any syzygy of $\ol\calG$ can be lifted to a syzygy of $\calG$ since $I$ is saturated.
\end{proof}

\subsection{... and reduced Gr\"obner bases}

We keep the setting of Lemma~\ref{lem: approximation is useful}. 
If we can find generators for $\tilde J^{\ec}$ (defined in Equation \eqref{eq:tilde J ec}) whose reduction modulo~$\pi^M$ are expressible only using the variables $x_i$, then Proposition \ref{pro: aux 1; lem: approximation is useful}, together with Lemma~\ref{lem: approximation is useful}, tells us there is a set of generators for $I^{\ec}$ with the same reductions modulo~$\pi^M$, thus solving Problem \ref{prob:compute_saturation}. It might happen, though, that $\tilde J^{\ec}$ does not have a set of generators with this property, while $J^{\ec}$ does, e.g.,  if $\tilde J^{\ec} = (1+ \pi y_1)$, which gives $J^{\ec} = (1)$. 

To really work with $J^{\ec} \subset R\langle x_{1},  \ldots,  x_{n}, \, y_{1},  \ldots,  y_{r} \rangle$, we can use Gr\"obner bases in the context of convergent power series, as defined in \cite{MR4007445}. 
In loc.\ cit., for any monomial order $<_\omega$ (a well-order on monomials compatible with multiplication) the authors define an order $<$ on terms $a\cdot m$, with $a\in K^\times$ and $m$ a monomial, by first looking at the valuation of $a$ and then at the monomial $m$; in particular $am < a'm'$ either if $v(a) > v(a')$ or if $v(a)=v(a')$ and $m <_\omega m'$. In particular, convergent power series $f$ always have a leading (i.e. maximal) term $\LT(f)$ and a Gr\"obner basis of an ideal $I$ is a set of elements $f_i \in I$ such that, for each $f\in I$, the leading term $\LT(f)$ is a multiple of at least one of the $\LT(f_i)$'s. We refer to the original article for properties and algorithms, and in particular to \cite[Algorithms 1 and 2]{MR4007445}, namely the analogues of division with remainder and  Buchberger's algorithm.

To achieve uniqueness we give the following definition.
\begin{Definition}
\label{def: reduce GB}
    A Gr\"obner basis $\{g_1, \ldots, g_u\}$ of an ideal $I \subset \calR_n$ is \emph{reduced} if, for each $i$, the series $g_i$ is monic (i.e. $\LT(g_i) = 1 \cdot m$ with $m$ a monomial), no term of $g_i$ is a multiple of $\LT(g_j)$ for $j\neq i$, and moreover no term of $g_i-\LT(g_i)$ is a multiple of $\LT(g_i)$. 
\end{Definition}

Notice that the last condition is automatic in the usual polynomial context, since if one monomial divides another, then the latter is bigger. This is not always true in the context of convergent power series: for example, in the series $f = x + \pi xy$, the term $\pi xy$ is multiple of the term $x$, even though we have $\pi xy<x$ (no matter which is the monomial order we started with). Notice, also, that this ``problem'' can be solved by dividing $f$ by the unit  $(1+\pi y)$.

A reduced Gr\"obner basis is always minimal, in the sense of \cite{MR4007445}, i.e., the leading terms of the $g_i$'s are the skeleton  of $\LT(I)$ (see Definition 2.9 in loc.\ cit.), or, equivalently, the leading terms of the $g_i$'s do not divide each other. 
This easily implies the uniqueness of reduced Gr\"obner bases. Suppose $\calG_1, \calG_2$ are two such distinct bases of the same ideal; then by the previous observation, for all $f \in \calG_{1}$ there is a $g \in \calG_{2}$ with the same leading term, and conversely. If we take $f$ in $ \calG_1 \setminus \calG_{2}$, then $f - g$ is non-zero and in particular its leading term is a multiple of one of the leading terms in $\LT(\calG_1)=\LT(\calG_2)$. Anyway, since the leading terms of $f$ and $g$ cancel out, then $\LT(f-g)$ is either a term of $f-\LT(f)$ or a term of $g-\LT(g)$, contradicting the reducedness.
Notice that, without the last condition in Definition~\ref{def: reduce GB}, uniqueness would not be guaranteed, consider for example the ideal  $(x_{1}, \, x_{2}) = (x_{1}, \, x_{2} + \pi x_{2})$.

We now see how to find the reduced Gr\"obner basis of a saturated ideal $I \subset \calR_n$ starting from a minimal Gr\"obner basis $g_1, \ldots, g_u$ by applying the division algorithm \cite[Algorithm~1]{MR4007445}. 
We start ``reducing'' $g_1$ as follows: we apply the  algorithm to $g_1 - \LT(g_1)$, dividing it by $g_1, g_2, \ldots, g_u$ and finding power series $q_1, \ldots, q_u, r_1$ such that 
\begin{equation}\label{eq:division_g1_LT}
g_1 - \LT(g_1) = q_1 g_1 + \ldots q_u g_u + r_1 ,
\end{equation}
with $r_1$ only containing terms that are not multiple of any $\LT(g_i)$. 
In particular, the power series 
$$g_1' := \LT(g_1) + r_1 ,$$
still belongs to the ideal $I$ and, together with $g_2, \ldots, g_u$, gives a minimal Gr\"obner basis of $I$, since $\LT(g_1') = \LT(g_1)$ by construction. Since $I$ is saturated and $g_1'$ is part of a minimal basis, we deduce that  $\LT(g_1) = c_1 \cdot m$, with $m$ a monomial and $c_1$ an invertible element in $R$, hence we can make $g_1'$ monic by dividing by $c_1$. 
Moreover, no term of $g_1' - \LT(g_1) = r_1$ is a multiple of $\LT(g_1')$ and $g_1'$ has no term multiple of $\LT(g_j)$ for $j \neq 1$ (it is true for the terms of $r_1$ by construction and, by minimality, also for $\LT(g_1)$). 
If we continue, and successively substitute $g_i$ with $\frac{1}{c_i}(\LT(g_i)+r_i)$, with the analogous $c_i$ and $r_i$, we find a reduced Gr\"obner basis of $I$.

We have then proved the following proposition.

\begin{Proposition}
    For each saturated ideal $I$ of $\calR_n$, there exists a unique reduced Gr\"obner basis for a fixed monomial order. Moreover, there exists an algorithm to compute such a basis, starting from a minimal Gr\"obner basis of $I$.
\end{Proposition}

It should be noted that Algorithm~1 in \cite{MR4007445} does not terminate in a finite amount of time, but, as observed in loc.\ cit., it is possible to compute the result (i.e., the power series $q_1, \ldots, q_u, r_1$ in \eqref{eq:division_g1_LT}) with a fixed precision $O(\pi^M)$ in a finite amount of time, starting from an input known with the same precision $O(\pi^M)$. Hence, we can compute the reduced Gr\"obner basis of a saturated ideal with any fixed precision in a finite amount of time, starting from a minimal Gr\"obner basis known with the same precision. 

Reduced Gr\"obner  bases are useful, in the context of Lemma~\ref{lem: approximation is useful}, because of the following proposition.

\begin{Proposition}\label{prop:reduced_basis_no_y}
    Suppose the valuation on $K$ is discrete, let $J = J^{\ec}$ be a saturated ideal of $R\langle x_1, \ldots, x_n, y_1, \ldots, y_r\rangle$ and $M$ be a positive integer. The ideal $J$ admits a set of generators that, reduced modulo $\pi^M$, does not depend on the variables $y_j$ if and only if the reduced Gr\"obner basis of $J$, reduced modulo $\pi^M$, does not depend on the variables $y_j$.
\end{Proposition}
\begin{proof}
    One implication is true, since by \cite[Proposition~$2.18$]{MR4007445} Gr\"obner bases are sets of generators.
    For the other implication, if $\calG$ is a set of generators of $J$ that does not depend on the variables $y_j$, then applying  \cite[Algorithm~$2$]{MR4007445} to $\calG$ yields a Gr\"obner basis with the same property and, after extracting a minimal basis and applying the reduction process described above (in particular \cite[Algorithm~$1$]{MR4007445}), the property is preserved. 
\end{proof}

The above proposition tells us that the condition on $J^{\ec} \bmod{\pi^M}$ in Lemma~\ref{lem: approximation is useful} can be checked using the reduced  Gr\"obner basis, which we know how to compute. 
This gives an algorithm for Problem~\ref{prob:compute_saturation}: we can pass from the ideal $I \subset \calR_n$, which we know with finite approximation, to the ideal $J\subset \calR_{n+r}$ defined in Lemma~\ref{lem: approximation is useful}, then compute a reduced Gr\"obner basis $\calG$ for $J^{\ec}$ with precision $O(\pi^M)$, and, if the condition in the lemma is satisfied, return $\calG$. The lemma ensures that such an output is correct, that is that $\calG$ approximates a set of generators of $I^{\ec}$.  This is made precise in Algorithm \ref{alg: check condition; lem: approximation is useful}.

\begin{algorithm}[h]
\caption{Either FAIL or solve Problem \ref{prob:compute_saturation}} 
\label{alg: check condition; lem: approximation is useful}

    \SetKwInOut{Input}{Input}\SetKwInOut{Output}{Output}
    \SetKwFor{ForEach}{for each}{do}{end}

    \Input{A positive integer $M$ 
    and a set $\bigl\{p_1 + O(\pi^N), \ldots, p_r+ O(\pi^N) \bigr\}$ of generators of an ideal $I \subset \calR_n$.}
    \Output{FAIL or a set of generators of $I^{\ec}$  known with precision $O(\pi^M)$.  }

    \BlankLine

    Compute generators for the saturation $\tilde J^{\ec}$, inside $R[x_1, \ldots, x_n, y_1, \ldots, y_r]$, of the ideal $\tilde J = (p_1 + \pi^Ny_1, \ldots, p_r + \pi^Ny_r)$ (see \cite[Algorithm~$18$]{MR3295981});
    \\
    Choose a monomial order and compute, with precision $O(\pi^M)$, a Gr\"obner basis $\calG$ of $\tilde J^{\ec} \cdot R\langle x_1, \ldots, x_n, y_1, \ldots, y_r \rangle$ (see \cite[Algorithm~$2$]{MR4007445});
    \\
    \ForEach{$g \in \calG$}{
        if $g = O(\pi^M)$ or if $\LT(g)$ is multiple of $\LT(g')$ for some $g'\in \calG \setminus \{g\}$, then delete $g$ from $\calG$;
    } \,   
    \ForEach{$g \in \calG$}{
        Compute, with precision $O(\pi^M)$, the remainder $r_{g}$ of the division of $g-\LT(g)$ by $\calG$ (see \cite[Algorithm~$1$]{MR4007445}); 
        \\
        Compute the coefficient $c$ of the leading term of $g$;
        \\
        \If{$\bigl(\tfrac{1}{c}\big(\LT(g)+r_{g}\big) \bmod{\pi^M}\bigr)$ contains some variable $y_{i}$}
                {
            \Return{FAIL;}
        }
        \Else{
        Replace $g$ with $\tfrac{1}{c}\bigl(\LT(g)+r_{g}\bigr) + O(\pi^M)$ in $\calG$;
        }
    }     
    \Return{$\calG$;}
\end{algorithm}

In more detail, the algorithm proceeds as follows. Using an algorithm for polynomial rings, we compute generators for the ideal $\tilde J^{\ec}$ in Equation~\eqref{eq:tilde J ec}. These are also generators for $J^{\ec}$ (Proposition \ref{pro: aux 1; lem: approximation is useful}). Then, using  Algorithm~$2$ in \cite{MR4007445} (the analogue of Buchberger's algorithm), we compute a Gr\"obner basis for the power series ideal $J^{\ec}$  with precision $O(\pi^M)$;   
since we work with finite precision and by \cite[Theorem~3.8]{MR4007445},  
Buchberger's algorithm terminates after a finite number of steps and gives a correct result.

The next step in the algorithm is to extract a minimal Gr\"obner basis of $J^{\ec}$ from this Gr\"obner basis. We do so either by checking directly  the minimality condition or, when it is not possible, namely if a series $g$ is equal $O(\pi^M)$ and consequently we cannot compute its leading term, we just delete the series since, by the saturatedness of $J^{\ec}$, all elements of a minimal basis have leading term with valuation $0$. 

Finally, we reduce the Gr\"obner basis using the procedure described below Definition \ref{def: reduce GB}.
During the reduction process we check whether the condition in Lemma~\ref{lem: approximation is useful} is satisfied: if the reduction modulo~$\pi^M$ of some element in the reduced basis depends on some variable  $y_j$, then we cannot apply Lemma~\ref{lem: approximation is useful} and the algorithm fails; otherwise, we output the approximation of the reduced basis itself.

\numberwithin{equation}{section}
\section{An application to a Thue equation}
\label{section example}

\newcommand{\Norm}{\mathrm{Norm}}

In this final section, we show an application of Theorem~\ref{main theorem section introduction}, together with Algorithm~\ref{alg: generators Il},
solving  the Thue equation
\begin{equation}\label{eq:Thue}
	X^{5}-X^{4}Y+X^{3}Y^{2}-X^{2}Y^{3}-Y^{5}=1 
\end{equation}
over the integers using Skolem's method (see, e.g., \cite{skolem1934verfahren, StroekerTzanakis, Lombardo}). We recall that, despite Skolem's method being a fun application of $p$-adic analysis, there are faster and more general methods to solve Thue equations, see e.g. \cite{MR1412969, Waldschmidt}, and, for Thue equations in families, \cite{AMZ,Marzenta}.

Equation \ref{eq:Thue} is associated to the number field
\begin{equation*}
	K\coloneqq\Q[\alpha]/(\alpha^{5}-\alpha^{4}+\alpha^{3}-\alpha^{2}-1) ,
\end{equation*}
since \eqref{eq:Thue} is equivalent to 
\begin{equation}\label{eq:Thue2}
	\Norm_{K/\Q} (X-Y \alpha)=1 .
\end{equation}

Since $K$ has degree $5$ and the group of units of $\Z[\alpha]$ has rank $2$, it is possible to apply Skolem's method, which, for any chosen prime $p$, relates solutions of \eqref{eq:Thue} to zeros of $p$-adic power series. We choose $p=5$ and go through the method. 

We have performed computations using a MAGMA (\cite{MAGMA}) script which applies Skolem's method to Thue equations related to quintic number fields of signature $(1,2)$, assuming that the fundamental units are known. The script is attached to the arXiv version of this paper.

If $(x,y)$ is an integer solution of \eqref{eq:Thue2}, there exist two integers $n_{1},\,n_{2}$ such that
\begin{equation} \label{eq:Thue_with_units}
	x-y\alpha=\pm u_{1}^{n_{1}} u_{2}^{n_{2}} ,
\end{equation}
where 
\begin{equation*}
	u_{1}\coloneqq \alpha^{4}-\alpha^{3}+\alpha^{2}-\alpha \quad\text{and}\quad u_{2}\coloneqq \alpha^{3}-\alpha^{2}-1
\end{equation*}
are the fundamental units of $K$ (see \url{https://www.lmfdb.org/NumberField/5.1.2297.1} in the \cite{lmfdb} database).  
The idea of Skolem's method is to embed $\Z[\alpha]$ in $\Z_5[\alpha]:= \Z[\alpha]\otimes_{\Z}\Z_{5}$: in this ambient space, the integer solutions of \eqref{eq:Thue} lie in the intersection $B\,\cap\,\ol C$, where $B$ is the $\Z_5$-subspace 
spanned by $1$ and $\alpha$, and $\ol C$ is the $5$-adic closure 
of the set $C \coloneqq \Z[\alpha]^\times = \{\pm u_{1}^{n_{1}}u_{2}^{n_{2}} \,\colon\, n_{1},\,n_{2}\in\Z\}$. Since $C$ is parametrized by two variables, we expect $\ol C$ to have dimension at most $2$, hence the intersection $B\,\cap\,\ol C$ to be finite.

Given an integer solution of \eqref{eq:Thue}, there are, a priori, $25$ possibilities for the pair $(x \bmod 5, y \bmod 5)$; accordingly,  we separately study the $25$ intersections $B \cap \ol C \cap U_{a,\,b}$, where $U_{a,\,b}$ is the open subset of $\Z_5[\alpha]$ consisting of elements congruent to $a - b\alpha$ modulo~$5$. Some residue disks are not relevant since the congruence modulo~$5$ already shows that \eqref{eq:Thue} cannot hold for some pairs $(x \bmod 5, y \bmod 5)$. We are left with five pairs $(a,b)$ to analyze, namely 
$(1,0)$, $(0,4)$,  $(4,4)$, $(3,4)$, $(2,4)$.
    
We now focus on the case $(a,b) = (0,4)$.
We start by parametrizing the set 
$C \cap U_{0,4}$, i.e.,  the units of $\Z[\alpha]$ that are congruent to $\alpha$ modulo~$5$. The set of units congruent to $1$ modulo $5$ is a subgroup of $\Z[\alpha]^\times$ isomorphic to $\Z^2$, with generators 
\begin{align*}
	v_{1}&\coloneqq 5\alpha^{4}-5\alpha^{3}+5\alpha^{2}-5\alpha-4 , \\
	v_{2}&\coloneqq -17490060\alpha^{4}+27069090\alpha^{3}+18075\alpha^{2}+5549965\alpha-17135289 .
\end{align*}
Then, $C \cap U_{0,4}$ is the coset
\begin{equation}
	u \cdot \langle v_1, v_2\rangle = \{u v_1^{n_1} v_2^{n_2} \,:\, n_1, n_2 \in \Z \} ,
\end{equation}
with $u=9579030\alpha^{4}+17508135\alpha^{3}-11940095\alpha^{2}-17135289\alpha-17490060$.
Now, since both $v_1$ and $v_2$ are congruent to $1$ modulo $5$, the function $(n_1, n_2) \mapsto   u \cdot v_1^{n_1} v_2^{n_2}$ can be expressed as a power series in $n_1, n_2$. Indeed, denoting by $\exp$ and $\log$ the usual power series for the exponential and the logarithm, then $\exp(x)$ and $\log(1 + x)$ converge $5$-adically for $x$ multiple of $5$. In particular, we can compute $\log(v_1)$ and $\log(v_2)$, which are elements in $\Z_5[\alpha]$ that are congruent to $0$ modulo~$5$, and write 
\begin{equation} \label{eq:Cbar}
	u \cdot v_1^{n_1} v_2^{n_2} =u\cdot\exp\bigl(n_{1}\cdot \log(v_{1})+n_{2}\cdot\log(v_{2})\bigr) = \sum_{i=0}^{4} F_{i}(n_{1},n_{2}) \cdot \alpha^i
\end{equation}
for suitable convergent power series $F_{i} \in \Q_5\langle x_1, x_2\rangle$.  
Then, we have
\begin{equation*}
	C \cap B \cap U_{0,4}  = \left\{  u v_1^{n_1} v_2^{n_2} \,:\, n_1, n_2 \in \Z , \, F_{2}(n_{1},\,n_{2})=F_{3}(n_{1},\,n_{2})=F_{4}(n_{1},\,n_{2})=0 
	\right\} .
\end{equation*}
It follows that the set of pairs $(n_1, n_2)$ we look for is a subset of the common zeros of the power series $F_{2},\,F_{3},\,F_{4}$ in $(\Z_5)^2$. Notice that such common zeros correspond to the points in $\ol C \cap B \cap U_{0,4} $, since, when we look at $(n_1, n_2)$ in $\Z_5^2$, and not only in $\Z^2$, then formula \eqref{eq:Cbar} parameterizes $\ol C \cap U_{0,4}$. 
Without changing the notation, we multiply each $F_{i}$ by a suitable constant so that it has $5$-integral coefficients not all divisible by $5$ (this is possible since the $F_i$ are convergent power series). Then
\begin{align*}
	F_{2} &= 108 + 3t_{2} + 5( 13 t_{1} + 9 t_{1}^{2} + 7 t_{1}t_{2} + 24 t_{2}^{2} ) + 5^2 ( 3t_{1}^{3} + 4 t_{1}^{2}t_{2}  + 3 t_{1}t_{2}^{2} + t_{2}^{3} )  + 5^3(\cdots) , \\
	F_{3} &= 61  + 81t_{2} +  5( t_{1}^{2} + 23 t_{1}t_{2} + 9 t_{2}^{2} ) +5^2 ( 2 t_{1}^{3} + 3 t_{1}^{2}t_{2} + t_{2}^{3} ) + 5^3(\cdots) , \\
	F_{4} &= 83  + 93t_{2} + 5( 12 t_{1} + 9 t_{1}^{2} + 16 t_{1}t_{2} + 7 t_{2}^{2} ) + 5^2 ( t_{1}^{3} + 3 t_{1}^{2}t_{2}  + t_{1}t_{2}^{2} + t_{2}^{3} )  + 5^3(\cdots) .
\end{align*}

Let $I$ be the ideal of $\Z_5\langle t_1,t_2\rangle$ generated by $F_2,F_3,F_4$ and let $\ol I$ be its reduction modulo~$5$. Since
\begin{equation*}
	\Fbb_{5}[t_{1},\,t_{2}]\bigl/ \ol I = \Fbb_{5}[t_{1},\,t_{2}]\bigl/\bigl(\ol{F_{2}},\,\ol{F_{3}},\,\ol{F_{4}}\bigr) = \Fbb_{5}[t_{1},\,t_{2}]\bigl/\bigl( t_2+1\bigr) \cong\Fbb_{5}[t_{1}]
\end{equation*}
is not finite, we cannot apply Theorem~\ref{main theorem section introduction}. Applying the first step of the saturation process (see Equation~\eqref{equation ideals Icerchietto n} and Algorithm~\ref{alg: generators Il}), we see that the ideal $I_1$, which has the same common zeros as $I$, is generated, modulo $5$, by $t_{2}+1$ and $t_{1}^{2}-t_{1}$. Since
\begin{equation*}
	\Fbb_{5}[t_{1},\,t_{2}]/(t_{2}+1,\,t_{1}^{2}-t_{1})\cong\Fbb_{5} \times \F_5 \text{,}
\end{equation*}
Theorem~\ref{main theorem section introduction} implies that the number of $5$-adic integer common zeros of the power series $F_{2},\,F_{3},\,F_{4}$ is at most $2$. It follows that the intersection $B\,\cap\,C\,\cap U_{0,\,4}$ contains at most two points, and in particular that \eqref{eq:Thue} has at most two solutions $(x,\,y)$ congruent to $(0,\,4)$ modulo~$5$. Indeed we know two solutions, which are $(0,\,-1)$ and $(5,\,4)$.

Repeating this process for each residue polydisk, we conclude that the number of integer solutions to \eqref{eq:Thue} is at most $4$. 
On the other hand,  we know four solutions, namely
$$(-1,-1), \, (0,-1), \, (1,0), \, (5,4).$$
Thus, these are all the solutions to our Thue equation.

\printbibliography[heading=bibintoc]

\end{document}